\documentclass[reqno]{amsart}
\usepackage{amssymb}

\newtheorem{theorem}{Theorem}[section]
\newtheorem{proposition}[theorem]{Proposition}
\newtheorem{lemma}[theorem]{Lemma}

\theoremstyle{definition}
\newtheorem*{definition}{Definition}

\theoremstyle{remark}
\newtheorem*{remark}{Remark}
\newtheorem*{note}{Note}

\numberwithin{equation}{section}

\begin{document}

\title[Bernstein-Szeg\"o Polynomials for Root Systems]
{Bernstein-Szeg\"o Polynomials Associated with Root Systems}

\author{J.F. van Diejen, A.C. de la Maza, and S. Ryom-Hansen}
\address{
Instituto de Matem\'atica y F\'{\i}sica,
Universidad de Talca,
Casilla 747, Talca, Chile}

\thanks{Work supported in part by the {\em Fondo Nacional de Desarrollo
Cient\'{\i}fico y Tecnol\'ogico (FONDECYT)} Grants \# 1051012, \#
1040896, and \# 1051024, by the {\em Anillo Ecuaciones Asociadas a
Reticulados} financed by the World Bank through the {\em Programa
Bicentenario de Ciencia y Tecnolog\'{\i}a}, and by the {\em Programa
Reticulados y Ecuaciones of the Universidad de Talca}.}


\date{November, 2006}

\begin{abstract}
We introduce multivariate generalizations of the Bernstein-Szeg\"o
polynomials, which are associated to the root systems of the complex
simple Lie algebras. The multivariate polynomials in question
generalize Macdonald's Hall-Littlewood polynomials associated with
root systems. For the root system of type $A_1$ (corresponding to
the Lie algebra ${\mathfrak sl} (2;\mathbb{C})$) the classic
Bernstein-Szeg\"o polynomials are recovered.
\end{abstract}

\maketitle

\section{Introduction}\label{sec1}
Recent years have revealed the birth of an elegant multivariate
generalization of the classical theory of (basic) hypergeometric
orthogonal polynomials based on the root systems of complex simple
Lie algebras
\cite{opd:some,hec-sch:harmonic,mac:orthogonal,dun-xu:orthogonal,mac:affine}.
The purpose of the present paper is to introduce---in the same
spirit---a multivariate generalization of the Bernstein-Szeg\"o
polynomials \cite{sze:orthogonal}.

By definition, Bernstein-Szeg\"o polynomials $p_\ell(x)$, $\ell
=0,1,2,\ldots$, are the trigonometric orthogonal polynomials
obtained by Gram-Schmidt orthogonalization of the Fourier-cosine
basis $m_\ell (x)=\exp (i\ell x)+\exp(-i\ell x)$, $\ell
=0,1,2,\ldots$, with respect to the inner product
\begin{subequations}
\begin{equation}
\langle m_\ell , m_k \rangle_\Delta = \frac{1}{2\pi} \int_0^{\pi}
m_\ell (x) \overline{m_k(x)}\Delta (x)\text{d} x,
\end{equation}
characterized by the nonnegative rational trigonometric weight
function of the form
\begin{equation}\label{weight}
\Delta (x) = \frac{|\delta (x)|^2}{c(x)c(-x)},\quad c(x)=
\prod_{m=1}^M (1+t_me^{-2ix}),
\end{equation}
\end{subequations}
where $\delta(x):= \exp (ix)-\exp(-ix)$. Here the parameters
$t_1,\ldots ,t_M$ are assumed to lie in the domain $(-1,1)\setminus
\{ 0\}$. A crucial property of the polynomials in question is
that---for sufficiently large degree $\ell$---they are given
explicitly by the compact formula \cite{sze:orthogonal}
\begin{subequations}
\begin{equation}\label{ef}
p_\ell (x)= \frac{1}{\mathcal{N}_\ell \delta(x)} \left(
c(x)e^{i(\ell +1)x}-c(-x)e^{-i(\ell+1) x} \right) , \qquad \ell \geq
M-1,
\end{equation}
where
\begin{equation}\label{ef2}
\mathcal{N}_\ell=
\begin{cases}
1-t_1\cdots t_M  &\text{if}\ \ell = M-1 ,\\
 1 &\text{if}\ \ell \geq M .
\end{cases}
\end{equation}
\end{subequations}
Furthermore, the quadratic norms of the corresponding
Bernstein-Szeg\"o polynomials are given by \cite{sze:orthogonal}
\begin{equation}\label{nf}
\langle p_\ell , p_\ell \rangle_\Delta = \mathcal{N}_\ell^{-1},
\qquad \ell \geq M-1 .
\end{equation}

The main result of this paper is a multivariate generalization of
these formulas associated with the root systems of the complex
simple Lie algebras (cf. Theorems \ref{form:thm}, \ref{orth:thm} and
\ref{norm:thm} below). The classical formulas in Eqs.
\eqref{ef}--\eqref{nf} are recovered from our results upon
specialization to the case of the Lie algebra ${\mathfrak sl}
(2;\mathbb{C})$ (corresponding to the root system $A_1$). Particular
instances of the multivariate Bernstein-Szeg\"o polynomials
discussed here have previously surfaced in Refs.
\cite{rui:factorized,die:asymptotic}, in the context of a study of
the large-degree asymptotics of Macdonald's multivariate basic
hypergeometric orthogonal polynomials related to root systems
\cite{mac:orthogonal,mac:affine}. The simplest examples of our
multivariate Bernstein-Szeg\"o polynomials---corresponding to weight
functions characterized by $c$-functions of degree $M=0$ and degree
$M=1$, respectively---amount to the celebrated Weyl characters and
to Macdonald's Hall-Littlewood polynomials associated with root
systems \cite{mac:spherical,mac:orthogonal}.

The paper is organized as follows. In Section \ref{sec2} the main
results are stated. The remainder of the paper, viz.  Sections
\ref{sec3}--\ref{sec5}, is devoted to the proofs.

\begin{note}
Throughout we will make extensive use of the language of root
systems. For preliminaries and further background material on root
systems the reader is referred to e.g. Refs.
\cite{bou:groupes,hum:introduction}.
\end{note}

\section{Bernstein-Szeg\"o polynomials for root systems}\label{sec2}
Let $\mathbf{E}$ be a real finite-dimensional Euclidian vector space
with scalar product $\langle \cdot , \cdot\rangle$, and let
$\mathbf{R}$ denote an irreducible crystallographic root system
spanning $\mathbf{E}$. Throughout it is assumed that $\mathbf{R}$ be
\textit{reduced}. We will employ the following standard notational
conventions for the dual root system
$\mathbf{R}^\vee:=\{\boldsymbol{\alpha}^\vee \mid
\boldsymbol{\alpha}\in \mathbf{R} \}$ (where
$\boldsymbol{\alpha}^\vee :=
2\boldsymbol{\alpha}/\langle\boldsymbol{\alpha} ,\boldsymbol{\alpha}
\rangle$), the root lattice
$\mathcal{Q}:=\text{Span}_{\mathbb{Z}}(\mathbf{R})$ and its
nonnegative semigroup
$\mathcal{Q}_+:=\text{Span}_{\mathbb{N}}(\mathbf{R}_+)$ generated by
the positive roots $\mathbf{R}_+$; the duals of the latter two
objects are given by the weight lattice $\mathcal{P}:=\{
\boldsymbol{\lambda}\in\mathbf{E}\mid \langle \boldsymbol{\lambda}
,\boldsymbol{\alpha}^\vee\rangle \in\mathbb{Z},\ \forall
\boldsymbol{\alpha}\in\mathbf{R} \}$ and its dominant integral cone
$\mathcal{P}_+ :=\{ \boldsymbol{\lambda}\in\mathcal{P}\mid \langle
\boldsymbol{\lambda} ,\boldsymbol{\alpha}^\vee\rangle
\in\mathbb{N},\ \forall \boldsymbol{\alpha}\in\mathbf{R}_+ \}$.
Finally, we denote by $W$ the Weyl group generated by the orthogonal
reflections $r_{\boldsymbol{\alpha}}:\mathbf{E}\to\mathbf{E}$,
$\boldsymbol{\alpha}\in\mathbf{R}$ in the hyperplanes perpendicular
to the roots (so for $\mathbf{x}\in\mathbf{E}$ one has that
$r_{\boldsymbol{\alpha}}
(\mathbf{x})=\mathbf{x}-\langle\mathbf{x},\boldsymbol{\alpha}^\vee\rangle
\boldsymbol{\alpha}$). Clearly $\| w\mathbf{x} \|^2=\langle
w\mathbf{x},w\mathbf{x}\rangle=\langle
\mathbf{x},\mathbf{x}\rangle=\| \mathbf{x}\|^2$ for all
 $w\in W$ and $\mathbf{x}\in\mathbf{E}$.

The algebra $\boldsymbol{A}_{\mathbf{R}}$ of Weyl-group invariant
trigonometric polynomials on the torus
$\mathbb{T}_{\mathbf{R}}=\mathbf{E}/(2\pi \mathcal{Q}^\vee)$ (where
$\mathcal{Q}^\vee:=\text{Span}_{\mathbb{Z}}(\mathbf{R}^\vee)$) is
spanned by the basis of the symmetric monomials
\begin{equation}
m_{\boldsymbol{\lambda}} (\mathbf{x}) =
\frac{1}{|W_{\boldsymbol{\lambda}}|} \sum_{w\in W} e^{i\langle
\boldsymbol{\lambda} ,\mathbf{x}_w\rangle }, \qquad
\boldsymbol{\lambda}\in\mathcal{P}_+ .
\end{equation}
Here $|W_{\boldsymbol{\lambda}}|$ represents the order of stabilizer
subgroup $W_{\boldsymbol{\lambda}} := \{ w\in W\mid
w(\boldsymbol{\lambda} )=\boldsymbol{\lambda} \}$ and
$\mathbf{x}_w:=w(\mathbf{x})$. We endow
$\boldsymbol{A}_{\mathbf{R}}$ with the following inner product
structure
\begin{subequations}
\begin{equation}
\langle f,g\rangle_{\Delta}=  \frac{1}{|W|\text{Vol}(\mathbb{T}_{\mathbf{R}})}
\int_{\mathbb{T}_{\mathbf{R}}} f(\mathbf{x})\overline{g(\mathbf{x})}
\Delta (\mathbf{x})\text{d}\mathbf{x} \qquad (f,g\in \boldsymbol{A}_{\mathbf{R}}),
\end{equation}
associated to a $W$-invariant nonnegative weight function that
factorizes over the root system:
\begin{eqnarray}
\Delta (\mathbf{x}) &=& \frac{|\delta
(\mathbf{x})|^2}{C(\mathbf{x})C(-\mathbf{x})} ,\\
  \delta
(\mathbf{x})&= &\prod_{\boldsymbol{\alpha}\in\mathbf{R}_+} \bigl(
e^{i\langle \boldsymbol{\alpha}
,\mathbf{x}\rangle /2}-e^{-i\langle \boldsymbol{\alpha} ,\mathbf{x}\rangle /2} \bigr),\\
 C(\mathbf{x}) &=&
\prod_{\boldsymbol{\alpha}\in\mathbf{R}^{(s)}_+}\!\!
c^{(s)}(e^{-i\langle \boldsymbol{\alpha} ,\mathbf{x}\rangle})
\prod_{\boldsymbol{\alpha}\in\mathbf{R}^{(l)}_+}\!\!
c^{(l)}(e^{-i\langle \boldsymbol{\alpha} ,\mathbf{x}\rangle})
,\makebox[1em]{}\label{cf1}
\end{eqnarray}
where
\begin{equation}
 c^{(s)}(z)=\prod_{m=1}^{M^{(s)}}
(1+t_m^{(s)}z), \qquad c^{(l)}(z)=\prod_{m=1}^{M^{(l)}}
(1+t_m^{(l)}z), \label{cf2}
\end{equation}
\end{subequations}
and with the parameters $t_m^{(s)}$ ($m=1,\ldots ,M^{(s)}$) and
$t_m^{(l)}$ ($m=1,\ldots ,M^{(l)}$) taken from $(-1,1)\setminus \{
0\}$. Here $|W|$ denotes the order of the Weyl group $W$,
$\text{Vol}(\mathbb{T}_{\mathbf{R}}):=\int_{\mathbb{T}_{\mathbf{R}}}
\text{d}\mathbf{x}$, and $\mathbf{R}^{(s)}_+:=\mathbf{R}^{(s)}\cap
\mathbf{R}_+$, $\mathbf{R}^{(l)}_+:=\mathbf{R}^{(l)}\cap
\mathbf{R}_+$, where $\mathbf{R}^{(s)}$ and $\mathbf{R}^{(l)}$ refer
to the \textit{short roots} and the \textit{long roots} of
$\mathbf{R}$, respectively (with the convention that all roots are
short, say, if $\mathbf{R}$ is simply-laced).

The Bernstein-Szeg\"o polynomials associated to the root system
$\mathbf{R}$ are now defined as the polynomials obtained from the
symmetric monomials $m_{\boldsymbol{\lambda}} (\mathbf{x})$,
$\boldsymbol{\lambda}\in\mathcal{P}_+$ by projecting away the
components in the finite-dimensional subspace spanned by monomials
corresponding to dominant weights that are smaller than
$\boldsymbol{\lambda}$ in the (partial) \textit{dominance ordering}
\begin{equation}\label{do}
\boldsymbol{\mu}\preceq\boldsymbol{\lambda}\quad \text{iff}\quad
\boldsymbol{\lambda}-\boldsymbol{\mu}\in\mathcal{Q}_+.
\end{equation}

\begin{definition}
The \textit{(monic) Bernstein-Szeg\"o polynomials}
$p_{\boldsymbol{\lambda}} (\mathbf{x})$,
$\boldsymbol{\lambda}\in\mathcal{P}_+$ are the polynomials of the
form
\begin{subequations}
\begin{equation}\label{bs1}
 p_{\boldsymbol{\lambda}} (\mathbf{x}) = \sum_{\boldsymbol{\mu}\in\mathcal{P}^+,\, \boldsymbol{\mu}
 \preceq
\boldsymbol{\lambda}} a_{\boldsymbol{\lambda} \boldsymbol{\mu}}
m_{\boldsymbol{\mu}} (\mathbf{x}) ,
\end{equation}
with expansion coefficients $a_{\boldsymbol{\lambda}
\boldsymbol{\mu}} \in\mathbb{C}$ such that $a_{\boldsymbol{\lambda}
\boldsymbol{\lambda}}=1$ and
\begin{equation}\label{bs2}
\langle p_{\boldsymbol{\lambda}} , m_{\boldsymbol{\mu}}
\rangle_{\Delta} = 0\quad \text{for}\ \boldsymbol{\mu} \prec
\boldsymbol{\lambda} .
\end{equation}
\end{subequations}
\end{definition}

It is clear that for any $\boldsymbol{\lambda}\in\mathcal{P}_+$ the
properties in Eqs. \eqref{bs1}, \eqref{bs2} determine
$p_{\boldsymbol{\lambda}} (\mathbf{x})$ uniquely. The main result of
this paper is an explicit formula for the Bernstein-Szeg\"o
polynomials for weights $\boldsymbol{\lambda}$ sufficiently deep in
the dominant cone $\mathcal{P}_+$. To formulate the precise result
we introduce the quantities
\begin{equation}
m^{(s)}(\boldsymbol{\lambda} ) =
\min_{\boldsymbol{\alpha}\in\mathbf{R}^{(s)}_+} \langle
\boldsymbol{\lambda} ,\boldsymbol{\alpha}^\vee\rangle \quad
\text{and}\quad m^{(l)}(\boldsymbol{\lambda} ) =
\min_{\boldsymbol{\alpha}\in\mathbf{R}^{(l)}_+} \langle
\boldsymbol{\lambda} ,\boldsymbol{\alpha}^\vee\rangle ,
\end{equation}
which measure the distance of the dominant weight
$\boldsymbol{\lambda}$ to the walls $\{
\boldsymbol{\mu}\in\mathcal{P}_+\mid \exists
\boldsymbol{\alpha}\in\mathbf{R}_+\ \text{such\ that}\ \langle
\boldsymbol{\mu},\boldsymbol{\alpha}^\vee\rangle =0\}$ bounding the
dominant cone. For future reference we will also single out the
special dominant weights given by the half-sums of the positive
roots:
\begin{equation}\label{rho}
\boldsymbol{\rho}:=\frac{1}{2}\sum_{\boldsymbol{\alpha}\in\mathbf{R}_+}\boldsymbol{\alpha},
\qquad
\boldsymbol{\rho}^{(s)}:=\frac{1}{2}\sum_{\boldsymbol{\alpha}\in\mathbf{R}_+^{(s)}}\boldsymbol{\alpha}
, \qquad
\boldsymbol{\rho}^{(l)}:=\frac{1}{2}\sum_{\boldsymbol{\alpha}\in\mathbf{R}_+^{(l)}}\boldsymbol{\alpha}
.
\end{equation}

\begin{definition}
Let us call a weight $\boldsymbol{\lambda}\in\mathcal{P}_+$
\textit{sufficiently deep} in the dominant cone iff
\begin{equation}
m^{(s)}(\boldsymbol{\lambda})\geq M^{(s)}-1\quad \text{and}\quad
m^{(l)}(\boldsymbol{\lambda})\geq M^{(l )}-1
\end{equation}
(where $M^{(s)}$ and $M^{(l)}$ refer to the degrees of $c^{(s)}(z)$
and $c^{(l)}(z)$ in Eq. \eqref{cf2}, respectively).
\end{definition}

\begin{theorem}[Explicit Formula]\label{form:thm}
For $\boldsymbol{\lambda}\in\mathcal{P}_+$ sufficiently deep in the
dominant cone, the monic Bernstein-Szeg\"o polynomial
$p_{\boldsymbol{\lambda}} (\mathbf{x})$ \eqref{bs1}, \eqref{bs2} is
given explicitly by
\begin{subequations}
\begin{equation}\label{exp}
p_{\boldsymbol{\lambda}} (\mathbf{x}) =
\mathcal{N}_{\boldsymbol{\lambda}}^{-1}P_{\boldsymbol{\lambda}}
(\mathbf{x})\quad\text{with}\quad P_{\boldsymbol{\lambda}}
(\mathbf{x}) =\frac{1}{ \delta (\mathbf{x})} \sum_{w\in W} (-1)^w
C(\mathbf{x}_w) e^{i\langle \boldsymbol{\rho}+\boldsymbol{\lambda}
,\mathbf{x}_w\rangle} ,
\end{equation}
where $C(\mathbf{x})$ is taken from Eqs. \eqref{cf1},\eqref{cf2} and
$(-1)^w:=\det (w)$. Here the normalization constant is of the form
\begin{equation}\label{lc}
\mathcal{N}_{\boldsymbol{\lambda}} =
\prod_{\begin{subarray}{c}\boldsymbol{\alpha}\in\mathbf{R}_+^{(s)} \\
\langle \tilde{\boldsymbol{\lambda}},\boldsymbol{\alpha}^\vee\rangle
= 0
\end{subarray}}
\frac{1-\mathbf{t}_{s}^{1+\text{ht}_{s}(\boldsymbol{\alpha})}
\mathbf{t}_{l}^{\text{ht}_{l}(\boldsymbol{\alpha})}}
{1-\mathbf{t}_{s}^{\text{ht}_{s}(\boldsymbol{\alpha})}
\mathbf{t}_{l}^{\text{ht}_{l}(\boldsymbol{\alpha})}}
\prod_{\begin{subarray}{c}\boldsymbol{\alpha}\in\mathbf{R}_+^{(l)} \\
\langle \tilde{\boldsymbol{\lambda}},\boldsymbol{\alpha}^\vee\rangle
= 0
\end{subarray}}
\frac{1-\mathbf{t}_{s}^{\text{ht}_{s}(\boldsymbol{\alpha})}
\mathbf{t}_{l}^{1+\text{ht}_{l}(\boldsymbol{\alpha})}}
{1-\mathbf{t}_{s}^{\text{ht}_{s}(\boldsymbol{\alpha})}
\mathbf{t}_{l}^{\text{ht}_{l}(\boldsymbol{\alpha})}} ,
\end{equation}
\end{subequations}
where
$\tilde{\boldsymbol{\lambda}}:=\boldsymbol{\lambda}+\boldsymbol{\rho}
-M^{(s)}\boldsymbol{\rho}^{(s)}-M^{(l)}\boldsymbol{\rho}^{(l)}$,
$\mathbf{t}_s:=-t_1^{(s)}\cdots t_{M^{(s)}}^{(s)}$,
$\mathbf{t}_l:=-t_1^{(l)}\cdots t_{M^{(l)}}^{(l)}$,
$\text{ht}_{s}(\boldsymbol{\alpha}):=
\sum_{\boldsymbol{\beta}\in\mathbf{R}_+^{(s)}} \langle
\boldsymbol{\alpha},\boldsymbol{\beta}^\vee \rangle /2$ and
$\text{ht}_{l}(\boldsymbol{\alpha}):=
\sum_{\boldsymbol{\beta}\in\mathbf{R}_+^{(l)}} \langle
\boldsymbol{\alpha},\boldsymbol{\beta}^\vee \rangle /2$ (and with
the convention that empty products are equal to {\em one}).
\end{theorem}

It is immediate from the definition that the Bernstein-Szeg\"o
polynomials are orthogonal when corresponding to weights that are
comparable in the dominance ordering \eqref{do}. The following
theorem states that the orthogonality holds in fact also for
non-comparable weights, assuming at least one of them lies
sufficiently deep in the dominant cone.

\begin{theorem}[Orthogonality]\label{orth:thm}
When at least one of
$\boldsymbol{\lambda},\boldsymbol{\mu}\in\mathcal{P}_+$ lies
sufficiently deep in the dominant cone, the Bernstein-Szeg\"o
polynomials \eqref{bs1},\eqref{bs2} are orthogonal
\begin{equation}\label{orth:eq}
\langle p_{\boldsymbol{\lambda}}
,p_{\boldsymbol{\mu}}\rangle_{\Delta} = 0 \quad\text{if}\quad
\boldsymbol{\mu} \neq \boldsymbol{\lambda} .
\end{equation}
\end{theorem}

Our final result provides an explicit formula for the quadratic
norm of the Bernstein-Szeg\"o polynomials corresponding to weights
sufficiently deep in the dominant cone.

\begin{theorem}[Norm Formula]\label{norm:thm}
For $\boldsymbol{\lambda}\in\mathcal{P}_+$ sufficiently deep in the
dominant cone, the quadratic norm of the monic Bernstein-Szeg\"o
polynomial \eqref{bs1},\eqref{bs2} is given by
\begin{equation}
\langle p_{\boldsymbol{\lambda}}
,p_{\boldsymbol{\lambda}}\rangle_{\Delta} =
\mathcal{N}_{\boldsymbol{\lambda}}^{-1}
\end{equation}
(with $\mathcal{N}_{\boldsymbol{\lambda}}$ given by Eq. \eqref{lc}).
\end{theorem}

For $M^{(s)}=M^{(l)}=0$ the above Bernstein-Szeg\"o polynomials boil
down to the Weyl characters
$\chi_{\boldsymbol{\lambda}}(\mathbf{x})$,
$\boldsymbol{\lambda}\in\mathcal{P}_+$ of the irreducible
representations of simple Lie algebras; and for $M^{(s)}=M^{(l)}=1$
they amount to Macdonald's Hall-Littlewood polynomials associated
with root systems. In these two simplest cases the contents of
Theorems \ref{form:thm}, \ref{orth:thm} and \ref{norm:thm} is
well-known from the representation theory of simple Lie algebras
\cite{hum:introduction} and from Macdonald's seminal work on the
zonal spherical functions on $p$-adic symmetric spaces
\cite{mac:spherical,mac:orthogonal}, respectively. Notice in this
connection that in these two special cases {\em all} dominant
weights are automatically sufficiently deep.

\begin{remark}{\em i.}
The weights $\boldsymbol{\lambda}$ sufficiently deep in the dominant
cone amount precisely to the {\em dominant} weights of the form
$\boldsymbol{\lambda}=\tilde{\boldsymbol{\lambda}}+(M^{(s)}-1)\boldsymbol{\rho}^{(s)}+(M^{(l)}-1)\boldsymbol{\rho}^{(l)}$
with $\tilde{\boldsymbol{\lambda}}\in\mathcal{P}_+$.
\end{remark}

\begin{remark}{\em ii.}
When  the dominant weights $\boldsymbol{\lambda}$,
$\boldsymbol{\mu}$ are not comparable in the dominance ordering
$\preceq $ \eqref{do} and moreover neither lies sufficiently deep in
the dominant cone, then there is no a priori reason for the
orthogonality in Eq. \eqref{orth:eq} to hold. Indeed, computer
experiments for small rank indicate that orthogonality may indeed be
violated in this situation. However, if one would replace in the
definition of the Bernstein-Szeg\"o polynomials given by Eqs.
\eqref{bs1}, \eqref{bs2} the dominance ordering by a {\em linear
ordering} that is compatible (i.e. extends) $\preceq$ \eqref{do},
then one would end up with an orthogonal basis that coincides with
our basis of Bernstein-Szeg\"o polynomials for weights
$\boldsymbol{\lambda}$ sufficiently deep. Clearly such a
construction would depend (where the weight is not sufficiently
deep) on the choice of the linear extension of the dominance
ordering $\preceq$ \eqref{do}.
\end{remark}

\begin{remark}{\em iii.}
The classical one-variable Bernstein-Szeg\"o polynomials play an
important role in the study of the large-degree asymptotics of
orthogonal polynomials on the unit circle \cite{sze:orthogonal}. In
a nutshell, the idea is that the weight function of the family of
orthogonal polynomials whose asymptotics one would like to determine
can be approximated (assuming certain analyticity conditions)  by
the weight function $\Delta (x)$ \eqref{weight} for $M\to +\infty$
and a suitable choice of the $t$-parameters. The explicit formula
for the Bernstein-Szeg\"o polynomials in Eqs. \eqref{ef},
\eqref{ef2} then converges to the asymptotic formula for the
orthogonal polynomials in question. In
\cite{rui:factorized,die:asymptotic}, special cases of the
multivariate Bernstein-Szeg\"o polynomials studied in the present
paper were employed to compute---in an analogous manner---the
asymptotics of families of multivariate orthogonal polynomials
(associated with root systems) for dominant weights
$\boldsymbol{\lambda}$ deep in the Weyl chamber (i.e. with the
distance to the walls going to $+\infty$). An important class of
multivariate polynomials whose large-degree asymptotics could be
determined by means of this method is given by the Macdonald
polynomials \cite{mac:orthogonal,mac:affine}.
\end{remark}

\section{Triangularity and Orthogonality}\label{sec3}
Following the spirit of Macdonald's concise approach towards the
Hall-Littlewood polynomials associated with root systems in Ref.
\cite[\S 10]{mac:orthogonal}, the idea of the proof of Theorem
\ref{form:thm} is to demonstrate that the explicit formula stated in
the theorem satisfies the two properties characterizing the
Bernstein-Szeg\"o polynomials given by Eqs. \eqref{bs1},
\eqref{bs2}. The orthogonality (Theorem \ref{orth:thm}) and the norm
formulas (Theorem \ref{norm:thm}) are then seen to follow from this
explicit formula.

First we verify the triangularity of $P_{\boldsymbol{\lambda}}
(\mathbf{x})$ \eqref{exp} with respect to the monomial basis
expressed in Eq. \eqref{bs1}.
\begin{proposition}[Triangularity]\label{tria:prp}
For $\boldsymbol{\lambda}\in\mathcal{P}_+$ sufficiently deep, the
expansion of the polynomial $P_{\boldsymbol{\lambda}} (\mathbf{x})$
\eqref{exp} on the monomial basis is triangular:
\begin{equation*}
P_{\boldsymbol{\lambda}} (\mathbf{x}) =
\sum_{\boldsymbol{\mu}\in\mathcal{P}^+,\, \boldsymbol{\mu} \preceq
\boldsymbol{\lambda}} c_{\boldsymbol{\lambda}\boldsymbol{\mu}}
m_{\boldsymbol{\mu}} (\mathbf{x}),
\end{equation*}
with $c_{\boldsymbol{\lambda}\boldsymbol{\mu}}\in\mathbb{C}$.
\end{proposition}
\begin{proof}
Upon expanding the products in $C(\mathbf{x})$ \eqref{cf1},
\eqref{cf2} it becomes evident that $P_{\boldsymbol{\lambda}}
(\mathbf{x})$ \eqref{exp} is built of a linear combination of
symmetric functions of the form
\begin{equation}\label{terms}
\delta^{-1}(\mathbf{x}) \sum_{w\in W} (-1)^w  e^{i\langle
\boldsymbol{\rho} +\boldsymbol{\lambda}
-\sum_{\boldsymbol{\alpha}\in \mathbf{R}^+} n_{\boldsymbol{\alpha}}
\boldsymbol{\alpha} , \mathbf{x}_w \rangle} ,
\end{equation}
with $0\leq n_{\boldsymbol{\alpha}} \leq M^{(s)}$ for
$\boldsymbol{\alpha}\in \mathbf{R}_+^{(s)}$ and $0\leq
n_{\boldsymbol{\alpha}} \leq M^{(l)}$ for $\boldsymbol{\alpha}\in
\mathbf{R}_+^{(l)}$. The expression in Eq. \eqref{terms} vanishes if
$\boldsymbol{\rho} +\boldsymbol{\lambda}
-\sum_{\boldsymbol{\alpha}\in \mathbf{R}^+} n_{\boldsymbol{\alpha}}
\boldsymbol{\alpha}$ is a singular weight and it is equal---possibly
up to a sign---to a Weyl character
$\chi_{\boldsymbol{\mu}}(\mathbf{x}):=\delta^{-1}(\mathbf{x})\sum_{w\in
W} (-1)^w e^{\langle
\boldsymbol{\rho}+\boldsymbol{\mu},\mathbf{x}_w\rangle }$ otherwise,
where $\boldsymbol{\mu}$ denotes the unique dominant weight in the
translated Weyl orbit $W(\boldsymbol{\rho} +\boldsymbol{\lambda}
-\sum_{\boldsymbol{\alpha}\in \mathbf{R}^+}n_{\boldsymbol{\alpha}}
\boldsymbol{\alpha})-\boldsymbol{\rho}$. Since $0\leq
n_{\boldsymbol{\alpha}} \leq M^{(s)}\leq
m^{(s)}(\boldsymbol{\lambda})+1=m^{(s)}(\boldsymbol{\rho}+\boldsymbol{\lambda})$
for $\boldsymbol{\alpha}\in \mathbf{R}_+^{(s)}$ and $0\leq
n_{\boldsymbol{\alpha}} \leq M^{(l)}\leq
m^{(l)}(\boldsymbol{\lambda})+1=m^{(l)}(\boldsymbol{\rho}+\boldsymbol{\lambda})$
for $\boldsymbol{\alpha}\in \mathbf{R}_+^{(l)}$, we can conclude
from Proposition \ref{saturated:prp} (below)---upon replacing
$\boldsymbol{\lambda}$ by
$\boldsymbol{\lambda}+\boldsymbol{\rho}$---that $\boldsymbol{\rho}
+\boldsymbol{\lambda} -\sum_{\boldsymbol{\alpha}\in
\mathbf{R}^+}n_{\boldsymbol{\alpha}} \boldsymbol{\alpha}\in
\mathcal{P}_+(\boldsymbol{\rho}+\boldsymbol{\lambda})$, whence
$\boldsymbol{\mu} \preceq \boldsymbol{\lambda}$. This shows that
$P_{\boldsymbol{\lambda}}(\mathbf{x})$ \eqref{exp} is a linear
combination of Weyl characters $\chi_{\boldsymbol{\mu}}(\mathbf{x})$
with $\boldsymbol{\mu}\preceq\boldsymbol{\lambda}$.  The statement
of the proposition is thus clear by the standard fact that Weyl
characters expand triangularly on the basis of monomial symmetric
functions.
\end{proof}

The next proposition checks (in particular) that
$P_{\boldsymbol{\lambda}} (\mathbf{x})$ \eqref{exp} satisfies the
orthogonality relations in Eq. \eqref{bs2}.

\begin{proposition}[Partial Biorthogonality]\label{orto:prp}
For $\boldsymbol{\lambda} ,\boldsymbol{\mu}\in\mathcal{P}_+$ such
that $\boldsymbol{\mu}\not\succ\boldsymbol{\lambda}$ the polynomial
$P_{\boldsymbol{\lambda}}(\mathbf{x})$ \eqref{exp} and the monomial
symmetric function $m_{\boldsymbol{\mu}}(\mathbf{x})$ satisfy the
orthogonality relations
\begin{equation*}
\langle P_{\boldsymbol{\lambda}}, m_{\boldsymbol{\mu }}
\rangle_\Delta =
\begin{cases}
0 & \text{if}\; \boldsymbol{\mu}\not\succeq\boldsymbol{\lambda} ,\\
1 & \text{if}\; \boldsymbol{\mu} =\boldsymbol{\lambda} .
\end{cases}
\end{equation*}
\end{proposition}
\begin{proof}
An explicit computation starting from the definitions entails that
\begin{eqnarray*}
&& \langle P_{\boldsymbol{\lambda}}, m_{\boldsymbol{\mu}}
\rangle_\Delta  = \frac{1 }{
|W|\,\text{Vol}(\mathbb{T}_{\mathbf{R}})\,
|W_{\boldsymbol{\mu}} |} \times \\
&& \int_{\mathbb{T}_{\mathbf{R}}} \frac{\delta
(-\mathbf{x})}{\mathcal{C}(\mathbf{x})\mathcal{C}(-\mathbf{x})}
\sum_{w_1\in W} (-1)^{w_1}\mathcal{C}(\mathbf{x}_{w_1}) e^{i\langle
\boldsymbol{\rho}+ \boldsymbol{\lambda}
,\mathbf{x}_{w_1}\rangle}\sum_{w_2\in W}e^{-i\langle
\boldsymbol{\mu}
,\mathbf{x}_{w_2}\rangle} \text{d}\mathbf{x} \\
&=&\!\! \frac{1 }{ \text{Vol}(\mathbb{T}_{\mathbf{R}})\,
|W_{\boldsymbol{\mu}} |} \sum_{w\in W}
\int_{\mathbb{T}_{\mathbf{R}}} \frac{1}{\mathcal{C}(-\mathbf{x})}
\prod_{\boldsymbol{\alpha}\in \mathbf{R}_+} (1-e^{i\langle
\boldsymbol{\alpha} ,\mathbf{x}\rangle})\; e^{i\langle
\boldsymbol{\lambda}-\boldsymbol{\mu}_w ,\mathbf{x}\rangle}
 \text{d}\mathbf{x} \\
&=& \!\!\frac{1}{\text{Vol}(\mathbb{T}_{\mathbf{R}})\,
|W_{\boldsymbol{\mu}} |}  \sum_{w\in W}
\int_{\mathbb{T}_{\mathbf{R}}}  e^{i\langle
\boldsymbol{\lambda}-\boldsymbol{\mu}_w ,\mathbf{x}\rangle}
\prod_{\boldsymbol{\alpha}\in \mathbf{R}_+} (1-e^{i\langle
\boldsymbol{\alpha} ,\mathbf{x}\rangle}) \times \\
&& \prod_{\boldsymbol{\alpha}\in
\mathbf{R}_+^{(s)}}(1+\sum_{n=1}^\infty f^{(s)}_n e^{i n \langle
\boldsymbol{\alpha} ,\mathbf{x}\rangle})
\prod_{\boldsymbol{\alpha}\in
\mathbf{R}_+^{(l)}}(1+\sum_{n=1}^\infty f^{(l)}_n e^{i n \langle
\boldsymbol{\alpha} ,\mathbf{x}\rangle}) \text{d}\mathbf{x} ,
\end{eqnarray*}
where  $f^{(s)}_n$ and $f^{(l)}_n$ denote the coefficients in the
Taylor series expansion of $1/c^{(s)}(z)$ and $1/c^{(l)}(z)$,
respectively, around $z=0$. The integrals on the last two lines pick
up the constant terms of the respective integrands multiplied by the
volume of the torus $\mathbb{T}_{\mathbf{R}}$. A nonzero constant
term can appear only when
$\boldsymbol{\mu}_w\succeq\boldsymbol{\lambda}$ (for some $w\in W$).
When $\boldsymbol{\mu}\not\succeq\boldsymbol{\lambda}$ one
automatically has that
$\boldsymbol{\mu}_w\not\succeq\boldsymbol{\lambda}$ for all $w\in W$
(since $\boldsymbol{\mu}_w\preceq\boldsymbol{\mu}$), whence the
constant term vanishes in this case. On the other hand, when
$\boldsymbol{\mu}=\boldsymbol{\lambda}$ the constant part of the
term labeled by $w$ is nonzero (namely equal to $1$) if and only if
$w\in W_{\boldsymbol{\lambda}}$. By summing over all these
contributions stemming from the stabilizer
$W_{\boldsymbol{\lambda}}$ the proposition follows.
\end{proof}

Combination of Propositions \ref{tria:prp} and \ref{orto:prp}
reveals that for $\boldsymbol{\lambda}\in\mathcal{P}_+$ sufficiently
deep $P_{\boldsymbol{\lambda}}(\mathbf{x})$ \eqref{exp} coincides
with the corresponding Bernstein-Szeg\"o polynomial
$p_{\boldsymbol{\lambda}}(\mathbf{x})$ defined by Eqs. \eqref{bs1},
\eqref{bs2} up to normalization. Furthermore, since it is clear from
Proposition \ref{orto:prp} and the definition of the
Bernstein-Szeg\"o polynomials that $\langle
P_{\boldsymbol{\lambda}}, p_{\boldsymbol{\mu}}\rangle_\Delta =0$ for
$\boldsymbol{\mu}\not\succeq \boldsymbol{\lambda}\in\mathcal{P}_+$,
we conclude that $\langle p_{\boldsymbol{\lambda}},
p_{\boldsymbol{\mu}}\rangle_\Delta =0$ for
$\boldsymbol{\mu}\not\succeq \boldsymbol{\lambda}\in\mathcal{P}_+$
with $\boldsymbol{\lambda}$ sufficiently deep; the orthogonality
stated in Theorem \ref{orth:thm} then follows in view of the
symmetry $\langle p_{\boldsymbol{\lambda}},
p_{\boldsymbol{\mu}}\rangle_\Delta =\overline{\langle
p_{\boldsymbol{\mu}}, p_{\boldsymbol{\lambda}}\rangle_\Delta}$.

\section{Normalization}\label{sec4}
To finish the proof of Theorem \ref{form:thm} it remains to verify
that the leading coefficient of
$P_{\boldsymbol{\lambda}}(\mathbf{x})$ \eqref{exp} is given by
$\mathcal{N}_{\boldsymbol{\lambda}}$ \eqref{lc}.

\begin{proposition}[Leading Coefficient]\label{lc:prp}
The leading coefficient in the monomial expansion of
$P_{\boldsymbol{\lambda}}(\mathbf{x})$ \eqref{exp} in Proposition
\ref{tria:prp} is given by
$c_{\boldsymbol{\lambda}\boldsymbol{\lambda}}=\mathcal{N}_{\boldsymbol{\lambda}}$
\eqref{lc}.
\end{proposition}
\begin{proof}
The polynomial $P_{\boldsymbol{\lambda}}(\mathbf{x})$ \eqref{exp}
reads explicitly
\begin{equation*}
\frac{1}{\delta (\mathbf{x})} \sum_{w\in W} (-1)^w e^{i\langle
\boldsymbol{\rho}+\boldsymbol{\lambda},\mathbf{x}_w\rangle}
\prod_{\boldsymbol{\alpha}\in\mathbf{R}_+^{(s)}}
\prod_{m=1}^{M^{(s)}} (1+t_m^{(s)} e^{-i\langle
\boldsymbol{\alpha},\mathbf{x}_w\rangle})
\prod_{\boldsymbol{\alpha}\in\mathbf{R}_+^{(l)}}
\prod_{m=1}^{M^{(l)}} (1+t_m^{(l)} e^{-i\langle
\boldsymbol{\alpha},\mathbf{x}_w\rangle}) .
\end{equation*}
As was remarked in the proof of Proposition \ref{tria:prp}, this
expression expands as a linear combination of terms of the form in
Eq. \eqref{terms}, with $0\leq n_{\boldsymbol{\alpha}} \leq
M^{(s)}\leq m^{(s)}(\boldsymbol{\rho}+\boldsymbol{\lambda})$ for
$\boldsymbol{\alpha}\in \mathbf{R}_+^{(s)}$ and $0\leq
n_{\boldsymbol{\alpha}} \leq M^{(l)}\leq
m^{(l)}(\boldsymbol{\rho}+\boldsymbol{\lambda})$ for
$\boldsymbol{\alpha}\in \mathbf{R}_+^{(l)}$. Upon replacing
$\boldsymbol{\lambda}$ by $\boldsymbol{\lambda}+\boldsymbol{\rho}$
in Proposition \ref{saturated:prp} and Proposition \ref{int:prp}
(below), it follows that in order for these terms to contribute to
the leading monomial it is {\em necessary} that
$n_{\boldsymbol{\alpha}} \in \{ 0, M^{(s)}\}$ for all
$\boldsymbol{\alpha}\in \mathbf{R}_+^{(s)}$ and $
n_{\boldsymbol{\alpha}} \in\{ 0, M^{(l)}\}$ for all
$\boldsymbol{\alpha}\in \mathbf{R}_+^{(l)}$. From now on we will
assume that both $M^{(s)}$ and $M^{(l)}$ are positive. (In the case
that $M^{(s)},M^{(l)}=0$ one has that
$P_{\boldsymbol{\lambda}}(\mathbf{x})=\chi_{\boldsymbol{\lambda}}(\mathbf{x})$,
whence $c_{\boldsymbol{\lambda}\boldsymbol{\lambda}}=1$ trivially;
the cases $M^{(s)}=0$, $M^{(l)}\geq 1$  and $M^{(s)}\geq 1$,
$M^{(l)}=0$ can be recovered from the analysis below upon
substituting formally $\mathbf{R}_+=\mathbf{R}_+^{(l)}$ and
$\mathbf{R}_+^{(s)}=\emptyset$ or $\mathbf{R}_+=\mathbf{R}_+^{(s)}$
and $\mathbf{R}_+^{(l)}=\emptyset$, respectively.) The corresponding
terms are then given explicitly by
\begin{equation*}
\frac{1}{\delta (\mathbf{x})} \sum_{w\in W} (-1)^w
\sum_{\mathbf{S}\subset\mathbf{R}_+} e^{i\langle
\boldsymbol{\mu}(\mathbf{S}),\mathbf{x}_w\rangle}
(-\mathbf{t}_s)^{\# (\mathbf{S}\cap \mathbf{R}_+^{(s)})}
(-\mathbf{t}_l)^{\# (\mathbf{S}\cap \mathbf{R}_+^{(l)})} ,
\end{equation*}
with $ \boldsymbol{\mu}
(\mathbf{S}):=\boldsymbol{\rho}+\boldsymbol{\lambda}-M^{(s)}\sum_{\boldsymbol{\alpha}\in\mathbf{S}\cap
\mathbf{R}_+^{(s)}}\boldsymbol{\alpha}-M^{(l)}\sum_{\boldsymbol{\alpha}\in\mathbf{S}\cap
\mathbf{R}_+^{(l)}}\boldsymbol{\alpha}$ and
$\mathbf{t}_s=-t_1^{(s)}\cdots t_{M^{(s)}}^{(s)}$,
$\mathbf{t}_l=-t_1^{(l)}\cdots t_{M^{(l)}}^{(l)}$. Rewriting this
expression in terms of Weyl characters
$\chi_{\boldsymbol{\mu}}(\mathbf{x})=\delta^{-1}(\mathbf{x})\sum_{w\in
W} (-1)^w e^{\langle
\boldsymbol{\rho}+\boldsymbol{\mu},\mathbf{x}_w\rangle }$,
$\boldsymbol{\mu}\in\mathcal{P}_+$ produces
\begin{equation*}
\sum_{\mathbf{S}\subset\mathbf{R}_+} (-1)^{w_{\mathbf{S}}}
\chi_{\boldsymbol{\lambda}(\mathbf{S})} (\mathbf{x})
(-\mathbf{t}_s)^{\# (\mathbf{S}\cap \mathbf{R}_+^{(s)})}
(-\mathbf{t}_l)^{\# (\mathbf{S}\cap \mathbf{R}_+^{(l)})} ,
\end{equation*}
where $w_{\mathbf{S}}$ denotes the unique shortest Weyl group
element permuting $\boldsymbol{\mu}(\mathbf{S})$ into the dominant
cone $\mathcal{P}_+$ and
$\boldsymbol{\lambda}(\mathbf{S}):=w_{\mathbf{S}}(\boldsymbol{\mu}(\mathbf{S}))-\boldsymbol{\rho}$
(here we have also assumed the convention that the Weyl character
$\chi_{\boldsymbol{\lambda} (\mathbf{S})} (\mathbf{x})$ vanishes
when $\boldsymbol{\lambda}(\mathbf{S})$ is not dominant). The
contributions to the leading monomial stem from those subsets
$\mathbf{S}\subset\mathbf{R}_+$ for which
$\boldsymbol{\lambda}(\mathbf{S})=\boldsymbol{\lambda}$, or
equivalently, $\boldsymbol{\mu}(\mathbf{S})\in
W(\boldsymbol{\rho}+\boldsymbol{\lambda})$. From Proposition
\ref{bound:prp} (below) with $\boldsymbol{\lambda}$ replaced by
$\boldsymbol{\lambda}+\boldsymbol{\rho}$, it follows that these are
precisely those subsets $\mathbf{S}\subset\mathbf{R}_+$ of the form
$\mathbf{S}=\mathbf{S}_w:=\{\boldsymbol{\alpha}\in\mathbf{R}_+\mid
w(\boldsymbol{\alpha})\not\in \mathbf{R}_+ \}$ for some $w\in
W_{\tilde{\boldsymbol{\lambda}}}$, where
$\tilde{\boldsymbol{\lambda}}:=\boldsymbol{\rho}+\boldsymbol{\lambda}-M^{(s)}
\boldsymbol{\rho}^{(s)}-M^{(l)}\boldsymbol{\rho}^{(l)}$ (cf. in this
connection also the remark just after Proposition \ref{bound:prp}).
By summing over all contributions from the subsets $\mathbf{S}_w$,
$w\in W_{\tilde{\boldsymbol{\lambda}}}$ (and recalling the fact that
the monomial expansion of the Weyl character
$\chi_{\boldsymbol{\lambda}}(\mathbf{x})$ is monic with leading term
$m_{\boldsymbol{\lambda}}(\mathbf{x})$), one concludes that the
leading coefficient $c_{\boldsymbol{\lambda}\boldsymbol{\lambda}}$
in the monomial expansion of $P_{\boldsymbol{\lambda}}(\mathbf{x})$
is given by the following Poincar\'e type series of the stabilizer
$W_{\tilde{\boldsymbol{\lambda}}}$:
\begin{equation*}
c_{\boldsymbol{\lambda}\boldsymbol{\lambda}}= \sum_{w\in
W_{\tilde{\boldsymbol{\lambda}}}} \mathbf{t}_s^{\ell_s (w)}
\mathbf{t}_l^{\ell_l(w)} ,
\end{equation*}
where $\ell_s(w):=\# \{\boldsymbol{\alpha}\in \mathbf{R}_+^{(s)}
\mid w(\boldsymbol{\alpha})\not\in \mathbf{R}_+^{(s)}\}$,
$\ell_l(w):=\# \{\boldsymbol{\alpha}\in \mathbf{R}_+^{(l)} \mid
w(\boldsymbol{\alpha})\not\in \mathbf{R}_+^{(l)}\}$. (Notice in this
respect that the minus signs dropped out as
$(-1)^{w_{\mathbf{S}}}=(-1)^{\ell_s(w_{\mathbf{S}})+\ell_l(w_{\mathbf{S}})}=(-1)^{\#
S}$.) Invoking a general product formula for the (two-parameter)
Poincar\'e series of Weyl groups due to Macdonald \cite[Theorem
(2.4)]{mac:poincare} then gives rise to
\begin{equation*}
c_{\boldsymbol{\lambda}\boldsymbol{\lambda}} =
\prod_{\begin{subarray}{c}\boldsymbol{\alpha}\in\mathbf{R}_+^{(s)} \\
\langle \tilde{\boldsymbol{\lambda}},\boldsymbol{\alpha}^\vee\rangle
= 0
\end{subarray}}
\frac{1-\mathbf{t}_{s}^{1+\text{ht}_{s}(\boldsymbol{\alpha})}
\mathbf{t}_{l}^{\text{ht}_{l}(\boldsymbol{\alpha})}}
{1-\mathbf{t}_{s}^{\text{ht}_{s}(\boldsymbol{\alpha})}
\mathbf{t}_{l}^{\text{ht}_{l}(\boldsymbol{\alpha})}}
\prod_{\begin{subarray}{c}\boldsymbol{\alpha}\in\mathbf{R}_+^{(l)} \\
\langle \tilde{\boldsymbol{\lambda}},\boldsymbol{\alpha}^\vee\rangle
= 0
\end{subarray}}
\frac{1-\mathbf{t}_{s}^{\text{ht}_{s}(\boldsymbol{\alpha})}
\mathbf{t}_{l}^{1+\text{ht}_{l}(\boldsymbol{\alpha})}}
{1-\mathbf{t}_{s}^{\text{ht}_{s}(\boldsymbol{\alpha})}
\mathbf{t}_{l}^{\text{ht}_{l}(\boldsymbol{\alpha})}} ,
\end{equation*}
where  $\text{ht}_{s}(\boldsymbol{\alpha})=
\sum_{\boldsymbol{\beta}\in\mathbf{R}_+^{(s)}} \langle
\boldsymbol{\alpha},\boldsymbol{\beta}^\vee \rangle /2$ and
$\text{ht}_{l}(\boldsymbol{\alpha})=
\sum_{\boldsymbol{\beta}\in\mathbf{R}_+^{(l)}} \langle
\boldsymbol{\alpha},\boldsymbol{\beta}^\vee \rangle /2$, which
completes the proof of the proposition.
\end{proof}

Finally, by combining Propositions \ref{tria:prp}, \ref{orto:prp},
and \ref{lc:prp}, the norm formula in Theorem \ref{norm:thm} readily
follows: $\langle
p_{\boldsymbol{\lambda}},p_{\boldsymbol{\lambda}}\rangle_\Delta=
\mathcal{N}_{\boldsymbol{\lambda}}^{-2}\langle
P_{\boldsymbol{\lambda}},P_{\boldsymbol{\lambda}}\rangle_\Delta=
\mathcal{N}_{\boldsymbol{\lambda}}^{-1}\langle
P_{\boldsymbol{\lambda}},m_{\boldsymbol{\lambda}}\rangle_\Delta =
\mathcal{N}_{\boldsymbol{\lambda}}^{-1}$ (for $\boldsymbol{\lambda}$
sufficiently deep in the dominant cone).

\section{Saturated sets of weights}\label{sec5}
In the proof of Propositions \ref{tria:prp} and \ref{lc:prp} we
exploited geometric properties of saturated subsets of the weight
lattice that are of interest in their own right independent of the
current application. To formulate these properties some additional
notation is required. To a dominant weight $\boldsymbol{\lambda}$,
we associated the following finite subsets of the weight lattice
\begin{equation}
\mathcal{P}_+(\boldsymbol{\lambda}) :=\{ \boldsymbol{\mu} \in
\mathcal{P}_+ \mid \boldsymbol{\mu} \preceq \boldsymbol{\lambda}  \}
,\qquad \mathcal{P}(\boldsymbol{\lambda}):=
\bigcup_{\boldsymbol{\mu}\in\mathcal{P}_+(\boldsymbol{\lambda})}
W(\boldsymbol{\mu}) . \label{sat}
\end{equation}
The subset $\mathcal{P} (\boldsymbol{\lambda})$ is {\em saturated},
i.e. for each $\boldsymbol{\mu} \in
\mathcal{P}(\boldsymbol{\lambda})$ and $\boldsymbol{\alpha}\in
\mathbf{R}$ the $\boldsymbol{\alpha}$-string through
$\boldsymbol{\mu}$ of the form $\{ \boldsymbol{\mu}-\ell
\boldsymbol{\alpha} \mid \ell =0,\ldots ,\langle \boldsymbol{\mu}
,\boldsymbol{\alpha}^\vee\rangle\} $ belongs to
$\mathcal{P}(\boldsymbol{\lambda})$
\cite{bou:groupes,hum:introduction}. Reversely, any saturated subset
of the weight lattice containing $\boldsymbol{\lambda}$ necessarily
contains the whole of $\mathcal{P}(\boldsymbol{\lambda})$
\eqref{sat}.

It is known from the representation theory of simple Lie algebras
that $\mathcal{P}(\boldsymbol{\lambda})$ \eqref{sat} lies inside the
convex hull of the Weyl-orbit through the highest weight vector
$\boldsymbol{\lambda}$. More precisely, we have the following
geometric characterization of $\mathcal{P}(\boldsymbol{\lambda})$
taken from Ref. \cite[Prop. 11.3, part a)]{kac:infinite}.
\begin{lemma}[\cite{kac:infinite}]\label{kac:lem}
The saturated set $\mathcal{P}(\boldsymbol{\lambda})$ \eqref{sat}
amounts to the points of the translated root lattice
$\boldsymbol{\lambda}+\mathcal{Q}$ inside the convex hull of the
Weyl-orbit $W(\boldsymbol{\lambda})$:
\begin{equation*}
\mathcal{P}(\boldsymbol{\lambda}) = \text{Conv}
(W(\boldsymbol{\lambda}))\cap (\boldsymbol{\lambda}+\mathcal{Q}) .
\end{equation*}
\end{lemma}
Since any dominant weight is maximal in its Weyl orbit (see e.g.
Ref. \cite[Sec. 13.2]{hum:introduction}), it is clear from this
lemma that all weights of $\mathcal{P}(\boldsymbol{\lambda})$
\eqref{sat} are obtained from $\boldsymbol{\lambda}$ via iterated
subtraction of positive roots. The following proposition provides
quantitative information on the number of times positive roots may
be subtracted from $\boldsymbol{\lambda}$ without leaving the convex
hull of $W(\boldsymbol{\lambda})$.

\begin{proposition}\label{saturated:prp}
Let $\boldsymbol{\lambda}\in\mathcal{P}_+$ and let
$n_{\boldsymbol{\alpha}}$, $\boldsymbol{\alpha}\in \mathbf{R}_+$ be
integers such that $0\leq n_{\boldsymbol{\alpha}}\leq
m^{(s)}(\boldsymbol{\lambda})$, $\forall
\boldsymbol{\alpha}\in\mathbf{R}_+^{(s)}$ and $0\leq
n_{\boldsymbol{\alpha}}\leq m^{(l)}(\boldsymbol{\lambda})$, $\forall
\boldsymbol{\alpha}\in\mathbf{R}_+^{(l)}$. Then one has that
\begin{equation*}
\boldsymbol{\lambda} -\sum_{\boldsymbol{\alpha}\in\mathbf{R}_+}
n_{\boldsymbol{\alpha}}\boldsymbol{\alpha} \in
\mathcal{P}(\boldsymbol{\lambda}) .
\end{equation*}
\end{proposition}
The proof of this proposition hinges on two lemmas.

\begin{lemma}\label{ch1:lem}
For any $\boldsymbol{\mu},\boldsymbol{\nu}\in\mathcal{P}_+$ the
following inclusion holds
\begin{equation*}
\boldsymbol{\mu}+\text{Conv}(W(\boldsymbol{\nu}))\subset
\text{Conv}(W(\boldsymbol{\mu}+\boldsymbol{\nu})) .
\end{equation*}
\end{lemma}
\begin{proof}
Clearly it suffices to show that
$\boldsymbol{\mu}+W(\boldsymbol{\nu})\subset
\text{Conv}(W(\boldsymbol{\mu}+\boldsymbol{\nu}))$. Since all
weights in $W(\boldsymbol{\mu})+W(\boldsymbol{\nu})$ are smaller
than or equal to $\boldsymbol{\mu}+\boldsymbol{\nu}$, it is evident
that the intersection of $W(\boldsymbol{\mu})+W(\boldsymbol{\nu})$
with the cone of dominant weights $\mathcal{P}_+$ is contained in
$\mathcal{P}_+(\boldsymbol{\mu}+\boldsymbol{\nu})$. We thus conclude
that $\boldsymbol{\mu}+W(\boldsymbol{\nu})\subset
W(\boldsymbol{\mu})+W(\boldsymbol{\nu})\subset
\mathcal{P}(\boldsymbol{\mu}+\boldsymbol{\nu})$. But then we have in
particular that $\boldsymbol{\mu}+W(\boldsymbol{\nu})\subset
\text{Conv}(W(\boldsymbol{\mu}+\boldsymbol{\nu}))$ in view of Lemma
\ref{kac:lem}., whence the inclusion stated in the lemma follows.
\end{proof}

\begin{lemma}\label{ch2:lem} Let $a,b\geq 0$ and let
$\boldsymbol{\rho}^{(s)}$, $\boldsymbol{\rho}^{(l)}$ be given by Eq.
\eqref{rho}. Then the convex hull of
$W(a\boldsymbol{\rho}^{(s)}+b\boldsymbol{\rho}^{(l)})$ reads
explicitly
\begin{eqnarray*}
\lefteqn{\text{Conv}(W(a\boldsymbol{\rho}^{(s)}+b\boldsymbol{\rho}^{(l)}))
=} && \\
&& \Bigl\{  a\sum_{\boldsymbol{\alpha}\in\mathbf{R}_+^{(s)}}
t_{\boldsymbol{\alpha}} \boldsymbol{\alpha} + b
\sum_{\boldsymbol{\alpha}\in\mathbf{R}_+^{(l)}}
t_{\boldsymbol{\alpha}} \boldsymbol{\alpha}\mid -{\textstyle
\frac{1}{2}}\leq t_{\boldsymbol{\alpha}}\leq {\textstyle
\frac{1}{2}}, \boldsymbol{\alpha}\in\mathbf{R}_+ \Bigr\} . \nonumber
\end{eqnarray*}
\end{lemma}
\begin{proof}
The r.h.s. is manifestly convex, Weyl-group invariant, and it
contains the vertex
$a\boldsymbol{\rho}^{(s)}+b\boldsymbol{\rho}^{(l)}$. We thus
conclude that the l.h.s. is a subset of the r.h.s. Furthermore, the
intersection of the l.h.s. with the closure of the dominant Weyl
chamber $\mathbf{C}:=\{ \mathbf{x}\in\mathbf{E} \mid \langle
\mathbf{x},\boldsymbol{\alpha}\rangle \geq 0,\; \forall
\boldsymbol{\alpha}\in\mathbf{R}_+\}$ consists of all vectors in
$\mathbf{C}$ that can be obtained from the vertex
$a\boldsymbol{\rho}^{(s)}+b\boldsymbol{\rho}^{(l)}$ by subtracting
nonnegative linear combinations of the positive roots. (This is
because the image of the vertex
$a\boldsymbol{\rho}^{(s)}+b\boldsymbol{\rho}^{(l)}$ with respect to
the orthogonal reflection in a wall of the dominant chamber is
obtained by subtracting a nonnegative multiple of the corresponding
simple root perpendicular to the wall in question.) Hence, the
intersection of $\mathbf{C}$ with the r.h.s. is contained in the
intersection of $\mathbf{C}$ with the l.h.s. But then the r.h.s.
must be a subset of the l.h.s. as both sides are Weyl-group
invariant (and the closure of the dominant Weyl chamber $\mathbf{C}$
constitutes a fundamental domain for the action of the Weyl group on
$\mathbf{E}$).
\end{proof}

To prove Proposition \ref{saturated:prp}, we apply Lemma
\ref{ch1:lem} with
\begin{equation}\label{wsub}
\boldsymbol{\mu}=\boldsymbol{\lambda}-
m^{(s)}\boldsymbol{\rho}^{(s)}-m^{(l)}\boldsymbol{\rho}^{(l)}\quad
\text{and}\quad \boldsymbol{\nu}=m^{(s)}\boldsymbol{\rho}^{(s)}+
m^{(l)}\boldsymbol{\rho}^{(l)},
\end{equation}
where $m^{(s)}=m^{(s)}(\boldsymbol{\lambda})$ and
$m^{(l)}=m^{(l)}(\boldsymbol{\lambda})$, respectively. Upon
computing $\text{Conv}(W(\boldsymbol{\nu}))$ with the aid of Lemma
\ref{ch2:lem} this entails the inclusion
\begin{eqnarray}\label{ci}
\boldsymbol{\lambda}- \Bigl\{ m^{(s)}
\sum_{\boldsymbol{\alpha}\in\mathbf{R}_+^{(s)}}
t_{\boldsymbol{\alpha}} \boldsymbol{\alpha} + m^{(l)}
\sum_{\boldsymbol{\alpha}\in\mathbf{R}_+^{(l)}}
t_{\boldsymbol{\alpha}} \boldsymbol{\alpha} \mid 0\leq
t_{\boldsymbol{\alpha}}\leq 1,
\boldsymbol{\alpha}\in\mathbf{R}_+\Bigr\}  &&\\
\subset \text{Conv}(W(\boldsymbol{\lambda})) , && \nonumber
\end{eqnarray}
which implies Proposition \ref{saturated:prp} in view of Lemma
\ref{kac:lem}.

The vertices of the convex set on the r.h.s. of Eq. \eqref{ci} are
given by the orbit $W(\boldsymbol{\lambda}) $, whereas for a point
to be a vertex of the convex set on the l.h.s. it is necessary that
$t_{\boldsymbol{\alpha}}\in \{ 0,1 \}$,
$\forall\boldsymbol{\alpha}\in\mathbf{R}_+$. This observation gives
rise to the following additional information regarding the weights
in Proposition \ref{saturated:prp} lying on the highest-weight orbit
$W(\boldsymbol{\lambda})$.

\begin{proposition}\label{int:prp}
Let $\boldsymbol{\lambda}\in\mathcal{P}_+$ and let
$n_{\boldsymbol{\alpha}}$, $\boldsymbol{\alpha}\in \mathbf{R}_+$ be
integers such that $0\leq n_{\boldsymbol{\alpha}}\leq
m^{(s)}(\boldsymbol{\lambda})$, $\forall
\boldsymbol{\alpha}\in\mathbf{R}_+^{(s)}$ and $0\leq
n_{\boldsymbol{\alpha}}\leq m^{(l)}(\boldsymbol{\lambda})$, $\forall
\boldsymbol{\alpha}\in\mathbf{R}_+^{(l)}$. Then $$
\boldsymbol{\lambda} -\sum_{\boldsymbol{\alpha}\in\mathbf{R}_+}
n_{\boldsymbol{\alpha}}\boldsymbol{\alpha} \in
W(\boldsymbol{\lambda}) $$ implies that $ n_{\boldsymbol{\alpha}}\in
\{ 0, m^{(s)}(\boldsymbol{\lambda})\}$, $\forall
\boldsymbol{\alpha}\in\mathbf{R}_+^{(s)}$ and $
n_{\boldsymbol{\alpha}}\in \{ 0, m^{(l)}(\boldsymbol{\lambda})\}$,
$\forall \boldsymbol{\alpha}\in\mathbf{R}_+^{(l)}$.
\end{proposition}

Much weaker versions of the statements in Proposition
\ref{saturated:prp} and Proposition \ref{int:prp} can be found in
the appendix of Ref. \cite{die:asymptotic}. For the root systems of
type $A$ Proposition \ref{saturated:prp} and a somewhat weaker form
of Proposition \ref{int:prp} were verified in Ref.
\cite{rui:factorized} by means of an explicit combinatorial
analysis.

Proposition \ref{int:prp} provides a necessary condition on the
coefficients $n_{\boldsymbol{\alpha}}$ such that a weight in
Proposition \ref{saturated:prp} lies on the highest-weight orbit
$W(\boldsymbol{\lambda})$. We will now wrap up with a more precise
characterization of the weights in question when
$\boldsymbol{\lambda}$ is strongly dominant.

\begin{lemma}\label{bound:lem}
For any $\boldsymbol{\mu},\boldsymbol{\nu}\in\mathcal{P}_+$ with
$\boldsymbol{\nu}$ strongly dominant (i.e. with
$m^{(s)}(\boldsymbol{\nu})$,  $m^{(l)}(\boldsymbol{\nu})$ strictly
positive), the intersection of
$\boldsymbol{\mu}+W(\boldsymbol{\nu})$ and
$W(\boldsymbol{\mu}+\boldsymbol{\nu})$ is given by
\begin{equation*}
(\boldsymbol{\mu}+W(\boldsymbol{\nu}))\cap
W(\boldsymbol{\mu}+\boldsymbol{\nu})=\boldsymbol{\mu}+
W_{\boldsymbol{\mu}}(\boldsymbol{\nu}).
\end{equation*}
\end{lemma}
\begin{proof}
The r.h.s. is manifestly contained in the intersection on the l.h.s.
It is therefore sufficient to demonstrate that the l.h.s. is also a
subset of the r.h.s.  The intersection on the l.h.s. consists of
those weights such that
$\boldsymbol{\mu}+w_1(\boldsymbol{\nu})=w_2(\boldsymbol{\mu}+\boldsymbol{\nu})$
for some $w_1, w_2\in W$. This implies that $\|
\boldsymbol{\mu}+w_1(\boldsymbol{\nu})\|^2=\|
\boldsymbol{\mu}+\boldsymbol{\nu}\|^2$, or equivalently,
$\langle\boldsymbol{\mu}-w_1^{-1}(\boldsymbol{\mu}),
\boldsymbol{\nu}\rangle =0$. But then we must have that
$w_1(\boldsymbol{\mu})=\boldsymbol{\mu}$ (and thus $w_2=w_1$) since
$\boldsymbol{\mu}-w_1^{-1}(\boldsymbol{\mu})\in\mathcal{Q}_+$ and
$\boldsymbol{\nu}$ is strongly dominant. It thus follows that the
weights in question form part of the r.h.s.
\end{proof}

By specializing Lemma \ref{bound:lem} to weights $\boldsymbol{\mu}$
and $\boldsymbol{\nu}$ of the form in Eq. \eqref{wsub} with
$\boldsymbol{\lambda}$ strongly dominant, and with $0<m^{(s)}\leq
m^{(s)}(\boldsymbol{\lambda})$ and $0<m^{(l)}\leq
m^{(l)}(\boldsymbol{\lambda})$, we arrive at the following
proposition.

\begin{proposition}\label{bound:prp}
Let $\boldsymbol{\lambda}\in\mathcal{P}_+$ be strongly dominant and
let $n_{\boldsymbol{\alpha}}$, $\boldsymbol{\alpha}\in \mathbf{R}_+$
be integers such that $0\leq n_{\boldsymbol{\alpha}}\leq m^{(s)}$,
$\forall \boldsymbol{\alpha}\in\mathbf{R}_+^{(s)}$ and $0\leq
n_{\boldsymbol{\alpha}}\leq m^{(l)}$, $\forall
\boldsymbol{\alpha}\in\mathbf{R}_+^{(l)}$, where $0<m^{(s)}\leq
m^{(s)}(\boldsymbol{\lambda})$ and $0<m^{(l)}\leq
m^{(l)}(\boldsymbol{\lambda})$. Then $ \boldsymbol{\lambda}
-\sum_{\boldsymbol{\alpha}\in\mathbf{R}_+}
n_{\boldsymbol{\alpha}}\boldsymbol{\alpha} \in
W(\boldsymbol{\lambda}) $ if and only if it is of the form
\begin{equation*}
\boldsymbol{\lambda} - m^{(s)} \sum_{\boldsymbol{\alpha}\in
\mathbf{S}_w\cap \mathbf{R}_+^{(s)} }\boldsymbol{\alpha} - m^{(l)}
\sum_{\boldsymbol{\alpha}\in \mathbf{S}_w\cap
\mathbf{R}_+^{(l)}}\boldsymbol{\alpha} ,
\end{equation*}
where $\mathbf{S}_w:=\{ \boldsymbol{\alpha}\in\mathbf{R}_+\mid
w(\boldsymbol{\alpha})\not\in\mathbf{R}_+ \}$ for some $w\in
W_{\tilde{\boldsymbol{\lambda}}}$ with
$$\tilde{\boldsymbol{\lambda}}:=\boldsymbol{\lambda}-
m^{(s)}\boldsymbol{\rho}^{(s)}-m^{(l)}\boldsymbol{\rho}^{(l)}.$$
\end{proposition}
\begin{proof}
The weights characterized by the premises of the proposition consist
of the common vertices of the convex sets on both sides of Eq.
\eqref{ci}. It is immediate from the previous discussion that the
vertices at issue are given by the weights in the intersection of
$\boldsymbol{\mu}+W(\boldsymbol{\nu})$ and
$W(\boldsymbol{\mu}+\boldsymbol{\nu})$, with $\boldsymbol{\mu}$ and
$\boldsymbol{\nu}$ given by Eq. \eqref{wsub}. According to Lemma
\ref{bound:lem}, this intersection consists of all weights of the
form
\begin{eqnarray*}
\boldsymbol{\lambda}-
m^{(s)}\boldsymbol{\rho}^{(s)}-m^{(l)}\boldsymbol{\rho}^{(l)}+
w^{-1}(m^{(s)}\boldsymbol{\rho}^{(s)}+
m^{(l)}\boldsymbol{\rho}^{(l)}) \\
= \boldsymbol{\lambda}- m^{(s)}\sum_{\begin{subarray}{c}
\boldsymbol{\alpha}\in\mathbf{R}_+^{(s)} \\
w(\boldsymbol{\alpha})\not\in \mathbf{R}_+^{(s)}\end{subarray}}
\boldsymbol{\alpha} - m^{(l)}\sum_{\begin{subarray}{c}
\boldsymbol{\alpha}\in\mathbf{R}_+^{(l)} \\
w(\boldsymbol{\alpha})\not\in \mathbf{R}_+^{(l)}\end{subarray}}
\boldsymbol{\alpha} ,
\end{eqnarray*}
where $w$ runs through the stabilizer of the weight
$\boldsymbol{\lambda}-
m^{(s)}\boldsymbol{\rho}^{(s)}-m^{(l)}\boldsymbol{\rho}^{(l)}$.
\end{proof}

\begin{remark}
Proposition \ref{bound:prp} implies Proposition \ref{int:prp} for
$\boldsymbol{\lambda}$ strongly dominant.  Indeed, the stabilizer
$W_{\tilde{\boldsymbol{\lambda}}}$ is generated by the reflections
in the short simple roots $\boldsymbol{\alpha}$ such that $\langle
\boldsymbol{\lambda},\boldsymbol{\alpha}^\vee\rangle = m^{(s)}$
(these reflections permute the roots of $\mathbf{R}^{(l)}_+$) and by
the reflections in the long simple roots $\boldsymbol{\alpha}$ such
that $ \langle \boldsymbol{\lambda},\boldsymbol{\alpha}^\vee\rangle
= m^{(l)}$ (these reflections permute the roots of
$\mathbf{R}^{(s)}_+$). Hence, for
$m^{(s)}<m^{(s)}(\boldsymbol{\lambda})$ or
$m^{(l)}<m^{(l)}(\boldsymbol{\lambda})$ one has that
$\mathbf{S}_w\cap\mathbf{R}_+^{(s)}=\emptyset$ or
$\mathbf{S}_w\cap\mathbf{R}_+^{(l)}=\emptyset$, respectively. In
other words, nonvanishing contributions to the sums in the formula
of Proposition \ref{bound:prp} only arise where $m^{(s)}$ or
$m^{(l)}$ assume their maximum values
$m^{(s)}(\boldsymbol{\lambda})$ and $m^{(l)}(\boldsymbol{\lambda})$,
respectively.
\end{remark}

\section*{Acknowledgments} Thanks are due to W. Soergel for indicating
the proof of Proposition \ref{saturated:prp}.

\bibliographystyle{amsplain}

\end{document}